\documentclass[12pt,a4paper,reqno]{amsart} 
\usepackage{amssymb}
\usepackage{amsthm}
\usepackage{latexsym}
\usepackage{amsmath}
\usepackage{mathrsfs}
\usepackage[X2,T1]{fontenc}
\usepackage{calc}                         
\usepackage{cite}

\newcommand{\masfR}{\mathsf R}
\newcommand{\scal}[2]{\langle #1,#2\rangle}

\newcommand{\rr}[1]{\mathbf R^{#1}}

\newcommand{\nm}[2]{\Vert #1\Vert _{#2}}
\newcommand{\nmm}[1]{\Vert #1\Vert }
\newcommand{\abp}[1]{\vert #1\vert}

\newcommand{\op}{\operatorname{Op}}

\newcommand{\sets}[2]{\{ \, #1\, ;\, #2\, \} }
\newcommand{\ep}{\varepsilon}

\newcommand{\cdo}{\, \cdot \, }

\newcommand{\wpr}{{\text{\footnotesize $\#$}}}

\newcommand{\BJ}{\operatorname{BJ}}
\newcommand{\vrum}{\vspace{0.1cm}}
\newcommand{\maclB}{\mathcal B}

\newcommand{\maclK}{\mathcal K}

\newcommand{\maclS}{\mathcal S}

\newcommand{\mascF}{\mathscr F}
\newcommand{\mascH}{\mathscr H}
\newcommand{\mascI}{\mathscr I}
\newcommand{\mascP}{\mathscr P}
\newcommand{\mascS}{\mathscr S}
\newcommand{\GL}{\mathbf{M}}
\newcommand{\UN}{\operatorname{UN}}
\newcommand{\mabfp}{\boldsymbol{p}}
\newcommand{\mabfq}{\boldsymbol{q}}
\newcommand{\sinc}{\operatorname{sinc}}

\numberwithin{equation}{section}          

\newtheorem{thm}{Theorem}
\numberwithin{thm}{section}

\newcommand{\rubrik}{}
\newtheorem{prop}[thm]{Proposition}

\newtheorem{lemma}[thm]{Lemma}

\theoremstyle{definition}

\newcommand{\rubrikdef}{}

\theoremstyle{remark}

\newtheorem{rem}[thm]{Remark}              

\author{Joachim Toft}

\address{Department of Mathematics, Linn{\ae}us University, Sweden}

\email{joachim.toft@lnu.se}

\title{Matrix parameterized pseudo-differential calculi on modulation spaces}



\begin{document}

\begin{abstract}
We consider a broad matrix parameterized family of pseudo-differential
calculi, containing the usual Shubin's family of pseudo-differential calculi,
parameterized by real numbers. We show that continuity properties
in the framework of modulation space theory, valid for the Shubin's family
extend to the broader matrix parameterized family of pseudo-differential
calculi.
\end{abstract}

\maketitle

\par

\section{Introduction}\label{sec0}

\par

A pseudo-differential calculus on $\rr d$ is a rule which takes any appropriate
function or distribution, defined on the phase space $T^*\rr d\asymp \rr {2d}$
into a set of linear operators acting on suitable functions or distributions defined
on $\rr d$. There are several other situations with similar approaches. For
example, a main issue in quantum mechanics concerns quantization, where
observables in classical mechanics (which are functions or distributions on
the phase space) carry over to corresponding observables in quantum
mechanics (which usually are linear operators on subspaces of $L^2(\rr d)$).
A somewhat similar situations can be found in time-frequency analysis.
Here the phase space corresponds to the time-frequency shift space, and
the filter parameters for (non-stationary filters) are suitable functions or
distributions on the time-frequency shift space, while the corresponding
filters are linear operators acting on signals (which are functions or
distributions, depending on the time).

\par

A common family of pseudo-differential calculi concerns $a\mapsto \op _t(a)$,
parameterized by $t\in \mathbf R$. If $a\in \mascS (\rr {2d})$, then
the pseudo-differential operator $\op _t(a)$ is defined by
$$
\op _t(a)f(x) = (2\pi )^{-d}\iint _{\rr {2d}} a(x-t(x-y),\xi )f(y)e^{i\scal {x-y}\xi}\, dyd\xi ,
$$
when $f\in \mascS (\rr d)$ (cf. e.{\,}g. \cite{Sh}).

\par

In the paper we consider as in \cite{CaTo} a slightly larger family of
pseudo-differential calculi, compared
to the situations above, which are parameterized by matrices instead of the real number
$t$. More precisely, if $a\in \mascS (\rr {2d})$ and $A$ is a real $d\times d$ matrix, then
the pseudo-differential operator $\op _A(a)$ is defined by
$$
\op _A(a)f(x) = (2\pi )^{-d}\iint _{\rr {2d}} a(x-A(x-y),\xi )f(y)e^{i\scal {x-y}\xi}\, dyd\xi ,
$$
when $f\in \mascS (\rr d)$ (cf. e.{\,}g. \cite{Sh}). We note that $\op _A(a)=\op _t(a)$
when $A=t\cdot I$, where $I$ is the $d\times d$ identity matrix. On the other hand,
in \cite{Bay}, D. Bayer consider a more general situation, where each pseudo-differential
calculus is parameterized by four matrices instead of one. 

\par

The definition of $\op _A(a)$ extends in several directions. In Section \ref{sec2} we
discuss such extensions within the theory of modulation spaces. That is, we deduce
continuity for such operators between different modulation spaces, when $a$ belongs 
to (other) modulation spaces. Similar analysis and results can be found in e.{\,}g.
\cite{Gc2,GH1,GH2,Sj1,To5,To8,To13} in the more restricted case $A=t\cdot I$, and
we emphasize that all results are obtained by using the framework of
these earlier contributions. Furthermore, some results here are in some cases
contained in certain results in Chapters 1 and 2 in \cite{Bay}.

\par

In Section \ref{sec3} we also give examples on how these operators might be used
in quantization, by taking the average of $\op _A(a)$ with $A=\frac 12\cdot I+r
\cdot U$, over all $r\in [0,1]$ and unitary matrices $U$ with real entries.

\par

\subsection*{Acknowledgement} I am very grateful to Ville Turunen for careful
reading and important advices, leading to several improvements of the original paper.

\par

\section{Preliminaries}\label{sec1}

\par

In this section we introduce some notations and discuss basic
results. We start by recalling some facts concerning
Gelfand-Shilov spaces. Thereafter we recall some properties about
pseudo-differential operators. Especially we discuss the Weyl
product and twisted convolution. Finally we recall some
facts about modulation spaces. The proofs are in general omitted,
since the results can be found in the literature.


\par

We start by considering Gelfand-Shilov spaces. 
Let $0<h,s\in \mathbf R$ be fixed. Then $\mathcal S_{s,h}(\rr d)$ consists of
all $f\in C^\infty (\rr d)$ such that
\begin{equation*}
\nm f{\mathcal S_{s,h}}\equiv \sup \frac {|x^\beta \partial ^\alpha
f(x)|}{h^{|\alpha | + |\beta |}\alpha !^s\, \beta !^s}
\end{equation*}
is finite. Here the supremum should be taken over all $\alpha ,\beta \in
\mathbf N^d$ and $x\in \rr d$.

\par

Obviously $\mathcal S_{s,h}\hookrightarrow
\mathscr S$ is a Banach space which increases with $h$ and $s$. Here and
in what follows we use the notation $A\hookrightarrow B$ when the topological
spaces $A$ and $B$ satisfy $A\subseteq B$ with continuous embeddings.
Furthermore, if $s>1/2$, or $s =1/2$ and $h$ is sufficiently large, then $\mathcal
S_{s,h}$ contains all finite linear combinations of Hermite functions.
Since such linear combinations are dense in $\mathscr S$, it follows
that the dual $(\mathcal S_{s,h})'(\rr d)$ of $\mathcal S_{s,h}(\rr d)$ is
a Banach space which contains $\mathscr S'(\rr d)$.

\par

The \emph{Gelfand-Shilov spaces} $\mathcal S_{s}(\rr d)$ and
$\Sigma _{s}(\rr d)$ are the inductive and projective limits respectively
of $\mathcal S_{s,h}(\rr d)$. This implies that
\begin{equation}\label{GSspacecond1}
\mathcal S_{s}(\rr d) = \bigcup _{h>0}\mathcal S_{s,h}(\rr d)
\quad \text{and}\quad \Sigma _{s}(\rr d) =\bigcap _{h>0}\mathcal S_{s,h}(\rr d),
\end{equation}
and that the topology for $\mathcal S_{s}(\rr d)$ is the strongest possible one such
that the inclusion map from $\mathcal S_{s,h}(\rr d)$ to $\mathcal S_{s}(\rr d)$
is continuous, for every choice of $h>0$. The space $\Sigma _s(\rr d)$ is a
Fr{\'e}chet space with semi norms
$\nm \cdo{\mathcal S_{s,h}}$, $h>0$. Moreover, $\Sigma _s(\rr d)\neq \{ 0\}$,
if and only if $s>1/2$, and $\maclS _s(\rr d)\neq \{ 0\}$
if and only if $s\ge 1/2$
(cf. \cite{GS,Ko,Pil}).

\medspace

The \emph{Gelfand-Shilov distribution spaces} $\mathcal S_{s}'(\rr d)$
and $\Sigma _s'(\rr d)$ are the projective and inductive limit
respectively of $\mathcal S_s'(\rr d)$.  This means that
\begin{equation}\tag*{(\ref{GSspacecond1})$'$}
\mathcal S_s'(\rr d) = \bigcap _{h>0}\mathcal S_{s,h}'(\rr d)\quad
\text{and}\quad \Sigma _s'(\rr d) =\bigcup _{h>0} \mathcal S_{s,h}'(\rr d).
\end{equation}
We remark that in \cite{Ko, Pil} it is proved that $\mathcal S_s'(\rr d)$
is the dual of $\mathcal S_s(\rr d)$, and $\Sigma _s'(\rr d)$
is the dual of $\Sigma _s(\rr d)$ (also in topological sense).

\par

For each $\ep >0$ and $s>1/2$ we have
\begin{equation}\label{GSembeddings}
\begin{alignedat}{3}
\maclS _{1/2}(\rr d) &\hookrightarrow &\Sigma _s (\rr d) &\hookrightarrow &
\maclS _s(\rr d) &\hookrightarrow \Sigma _{s+\ep}(\rr d)
\\[1ex]
\quad \text{and}\quad
\Sigma _{s+\ep}' (\rr d) &\hookrightarrow & \maclS _s'(\rr d)
&\hookrightarrow & \Sigma _s'(\rr d) &\hookrightarrow \maclS _{1/2}'(\rr d).
\end{alignedat}
\end{equation}

\par

The Gelfand-Shilov spaces are invariant under several basic transformations.
For example they are invariant under translations, dilations
and under (partial) Fourier transformations. We also note that the map
$(f_1,f_2)\mapsto f_1\otimes f_2$ is continuous from $\maclS _s(\rr {d_1})
\times \maclS _s(\rr {d_2})$ to $\maclS _s(\rr {d_1+d_2})$, and similarly when
each $\maclS _s$ are replaced by $\Sigma _s$, $\maclS _s'$ or by
$\Sigma _s'$.

\par

We let $\mathscr F$ be the Fourier transform which takes the form
$$
(\mathscr Ff)(\xi )= \widehat f(\xi ) \equiv (2\pi )^{-d/2}\int _{\rr
{d}} f(x)e^{-i\scal  x\xi }\, dx
$$
when $f\in L^1(\rr d)$. Here $\scal \cdo \cdo$ denotes the usual scalar product
on $\rr d$. The map $\mathscr F$ extends 
uniquely to homeomorphisms on $\mathscr S'(\rr d)$, $\mathcal S_s'(\rr d)$
and $\Sigma _s'(\rr d)$, and restricts to 
homeomorphisms on $\mathscr S(\rr d)$, $\mathcal S_s(\rr d)$ and $\Sigma _s(\rr d)$, 
and to a unitary operator on $L^2(\rr d)$.

\par

\par

\subsection{An extended family of pseudo-differential calculi}

\par

Next we discuss some issues in pseudo-differential calculus.
Let $\GL (d,\Omega)$ be the set of all $d\times d$-matrices with
entries in the set $\Omega$, and let $s\ge 1/2$, $a\in \maclS _s 
(\rr {2d})$ and $A\in \GL (d,\mathbf R)$ be fixed.
Then the pseudo-differential operator $\op _A(a)$ is the linear and
continuous operator on $\maclS _s (\rr d)$, given by
\begin{equation}\label{e0.5}
(\op _A(a)f)(x)
=
(2\pi  ) ^{-d}\iint a(x-A(x-y),\xi )f(y)e^{i\scal {x-y}\xi }\, dyd\xi ,
\end{equation}
when $f\in \maclS _s(\rr d)$. For
general $a\in \maclS _s'(\rr {2d})$, the
pseudo-differential operator $\op _A(a)$ is defined as the linear and
continuous operator from $\maclS _s(\rr d)$ to $\maclS _s'(\rr d)$ with
distribution kernel given by
\begin{equation}\label{atkernel}
K_{a,A}(x,y)=(2\pi )^{-d/2}(\mascF _2^{-1}a)(x-A(x-y),x-y).
\end{equation}
Here $\mascF _2F$ is the partial Fourier transform of $F(x,y)\in
\maclS _s'(\rr {2d})$ with respect to the $y$ variable. This
definition makes sense, since the mappings
\begin{equation}\label{homeoF2tmap}
\mascF _2\quad \text{and}\quad F(x,y)\mapsto F(x-A(x-y),x-y)
\end{equation}
are homeomorphisms on $\maclS _s'(\rr {2d})$.
In particular, the map $a\mapsto K_{a,A}$ is a homeomorphism on
$\maclS _s'(\rr {2d})$.

\par

An important special case appears when $A=t\cdot I$, with
$t\in \mathbf R$. Here and in what follows, $I\in \GL (d,\mathbf R)$ denotes
the $d\times d$ identity matrix. In this case we set
$$
\op _t(a) = \op _{t\cdot I}(a).
$$
The normal or Kohn-Nirenberg representation, $a(x,D)$, is obtained
when $t=0$, and the Weyl quantization, $\op ^w(a)$, is obtained
when $t=\frac 12$. That is,
$$
a(x,D) = \op _0(a)
\quad \text{and}\quad \op ^w(a) = \op _{1/2}(a).
$$

\par

For any $K\in \maclS '_s(\rr {d_1+d_2})$, we let $T_K$ be the
linear and continuous mapping from $\maclS _s(\rr {d_1})$
to $\maclS _s'(\rr {d_2})$, defined by the formula
\begin{equation}\label{pre(A.1)}
(T_Kf,g)_{L^2(\rr {d_2})} = (K,g\otimes \overline f )_{L^2(\rr {d_1+d_2})}.
\end{equation}
It is well-known that if $A\in \GL (d,\mathbf R)$, then it follows from Schwartz kernel
theorem that $K\mapsto T_K$ and $a\mapsto \op _A(a)$ are bijective
mappings from $\mascS '(\rr {2d})$
to the set of linear and continuous mappings from $\mascS (\rr d)$ to
$\mascS '(\rr d)$ (cf. e.{\,}g. \cite{Ho1}).

\par
 
Furthermore, by e.{\,}g. \cite[Theorem 2.2]{LozPerTask} it follows
that the same holds true if each $\mascS$ and $\mascS '$ are
replaced by $\maclS _s$ and $\maclS _s'$, respectively, or by
$\Sigma _s$ and $\Sigma _s'$, respectively.

\par

In particular, for every $a_1\in \maclS _s '(\rr {2d})$ and $A_1,A_2\in
\GL (d,\mathbf R)$, there is a unique $a_2\in \maclS _s '(\rr {2d})$ such that
$\op _{A_1}(a_1) = \op _{A_2} (a_2)$. The following result explains the
relations between $a_1$ and $a_2$.

\par

\begin{prop}\label{Prop:CalculiTransfer}
Let $a_1,a_2\in \maclS _{1/2}'(\rr {2d})$ and $A_1,A_2\in \GL (d,\mathbf R)$.
Then
\begin{equation}\label{calculitransform}
\op _{A_1}(a_1) = \op _{A_2}(a_2) \quad \Leftrightarrow \quad
e^{i\scal {A_2D_\xi}{D_x }}a_2(x,\xi )=e^{i\scal {A_1D_\xi}{D_x }}a_1(x,\xi ).
\end{equation}
\end{prop}

\par

Note here that the latter equality in \eqref{calculitransform} makes sense
since it is equivalent to
$$
e^{i\scal {A_2x}{\xi}}\widehat a_2(\xi ,x)
=e^{i\scal {A_1x}{\xi}}\widehat a_1(\xi ,x),
$$
and that the map $a\mapsto e^{i\scal {Ax} \xi }a$ is continuous on
$\maclS _s '$ (cf. e.{\,}g. \cite{Tr,CaTo}).

\par

Passages between different kinds of pseudo-differential calculi have been
considered before (Cf. e.{\,}g. \cite{Ho1,Tr}.) On the other hand, except for
\cite{CaTo}, it seems that the representation $a\mapsto \op _A(a)$ for general
matrix $A\in
\GL (d,\mathbf R)$, has not been considered in the literature before.

\par

A proof of Proposition \ref{Prop:CalculiTransfer} for matrices of the form $A=t\cdot I$,
$t\in \mathbf R$,
can be found in e.{\,}g. \cite{Sh}, and the proof of the result for general $A$
follows by similar arguments. In order to be self-contained we here present
the arguments. Here and in what follows, $A^*$ denotes the transpose of the matrix
$A$.

\par

\begin{proof}[Proof of Proposition \ref{Prop:CalculiTransfer}]
It is no restriction to assume that $A_2=0$.
Let $a=a_1$, $b=a_2$ and $A=A_1$.
We also prove the result only in the case
$a,b\in \mascS (\rr {2d})$. The general case follows by similar arguments and
is left for the reader.

\par

The equality $\op _A(a)=\op (b)$ is the same as
\begin{alignat*}{2}
\mascF (b(x,\cdo ))(y-x) &= \mascF (a(x-A(x-y),\cdo ))(y-x) &
\quad &\Leftrightarrow \quad
\\[1ex]
\int b(x,\eta )e^{i\scal y\eta }\, d\eta &= \int a(x-Ay,\eta )e^{i\scal y\eta}\, d\eta &
\quad &\Leftrightarrow \quad
\\[1ex]
b(x,\xi ) &= (2\pi )^{-d}\iint a(x-Ay,\eta )e^{i\scal y{\eta -\xi}}\, dy d\eta & &
\end{alignat*}
By Fourier's inversion formula we get
\begin{align*}
b(x,\xi )&= (2\pi )^{-d}\iint a(x-Ay,\eta )e^{i\scal y{\eta -\xi}}\, dy d\eta
\\[1ex]
&= (2\pi )^{-2d}\iiiint \widehat a(\eta _1,y_1 )
e^{i( \scal {x-Ay}{\eta _1}+\scal{y_1}{\eta} +\scal y{\eta -\xi})}\, dy_1 d\eta _1dy d\eta
\\[1ex]
&= (2\pi )^{-d}\iint \widehat a(\eta _1,y_1 )e^{i( \scal {x}{\eta _1}
+\scal{y_1}{\xi +A^*\eta _1})}\, dy_1 d\eta _1
\\[1ex]
&=
e^{i\scal{AD_\xi}{D_x}}a(x,\xi ),
\end{align*}
which gives the result.
\end{proof}

\par

Let $a\in \maclS _s '(\rr {2d})$ be
fixed. Then $a$ is called a rank-one element with respect to
$A\in \GL (d,\mathbf R)$, if $\op _A(a)$ is an operator of rank-one,
i.{\,}e.
\begin{equation}\label{trankone}
\op _A(a)f=(f,f_2)f_1, \qquad f\in \maclS _s(\rr d),
\end{equation}
for some $f_1,f_2\in \maclS _s '(\rr d)$. By
straight-forward computations it follows that \eqref{trankone}
is fulfilled if and only if $a=(2\pi
)^{\frac d2}W_{f_1,f_2}^A$, where $W_{f_1,f_2}^A$
is the $A$-Wigner distribution, defined by the formula
\begin{equation}\label{wignertdef}
W_{f_1,f_2}^A(x,\xi ) \equiv \mascF \big (f_1(x+A\cdo
)\overline{f_2(x+(A-I)\cdo )} \big ) (\xi ),
\end{equation}
which takes the form
$$
W_{f_1,f_2}^A(x,\xi ) =(2\pi )^{-\frac d2} \int
f_1(x+Ay)\overline{f_2(x+(A-I)y) }e^{-i\scal y\xi}\, dy,
$$
when $f_1,f_2\in \maclS _s (\rr d)$. By combining these facts
with \eqref{calculitransform}, it follows that
\begin{equation}\label{wignertransf}
e^{i\scal {A_2D_\xi  }{D_x}}W_{f_1,f_2}^{A_2}
=
e^{i\scal {A_1D_\xi }{D_x}} W_{f_1,f_2}^{A_1},
\end{equation}
for every $f_1,f_2\in \maclS _s '(\rr d)$ and $A_1,A_2\in \GL (d,\mathbf R)$. Since
the Weyl case is particularly important, we set
$W_{f_1,f_2}^{A}=W_{f_1,f_2}$ when $A=\frac 12I$, i.{\,}e.
$W_{f_1,f_2}$ is the usual (cross-)Wigner distribution of $f_1$ and
$f_2$.

\par

For future references we note the link
\begin{multline}\label{tWigpseudolink}
(\op _A(a)f,g)_{L^2(\rr d)} =(2\pi )^{-d/2}(a,W_{g,f}^A)_{L^2(\rr {2d})},
\\[1ex]
a\in \maclS _s'(\rr {2d}) \quad\text{and}\quad f,g\in \maclS _s(\rr d)
\end{multline}
between pseudo-differential operators and Wigner distributions,
which follows by straight-forward computations (see also e.{\,}g.
\cite{dG,dGLu}).

\medspace

For any $A\in \GL (d,\mathbf R)$, the
$A$-product, $a\wpr _Ab$ between $a\in \maclS _s' (\rr {2d})$
and $b\in \maclS _s'(\rr {2d})$ is defined by the formula
\begin{equation}\label{wprtdef}
\op _A(a\wpr _A b) = \op _A(a)\circ \op _A(b),
\end{equation}
provided the right-hand side makes sense as a continuous operator from
$\maclS _s (\rr d)$ to $\maclS _s '(\rr d)$.

\par

\subsection{Modulation spaces}

\par

Next we discuss basic properties for modulation spaces, and start by
recalling the conditions for the involved weight functions. A function $\omega$
on $\rr d$ is called a \emph{weight} (on $\rr d$), if $\omega >0$ and
$\omega ,\omega ^{-1}\in L^\infty _{loc}(\rr d)$. Let $\omega$ and $v$
be weights on $\rr d$. Then $\omega$ is called \emph{moderate} or
\emph{$v$-moderate} if
\begin{equation}\label{moderate}
\omega (x+y) \lesssim \omega (x)v(y),\quad x,y\in \rr d.
\end{equation}
Here and in what follows we write $A\lesssim B$ when $A,B\ge 0$
and $A\le cB$ for a suitable constant $c>0$. We also let
$A\asymp B$ when $A\lesssim B$ and $B\lesssim A$.
The weight $v$ is called \emph{submultiplicative},
if $v$ is even and \eqref{moderate} holds when
$\omega =v$. We note that if \eqref{moderate} holds, then
$$
v(-x)^{-1}\lesssim \omega (x) \lesssim v(x).
$$
Furthermore, for such $\omega$ it follows that \eqref{moderate} is true when
$$
v(x) =Ce^{c|x|},
$$
for some positive constants $c$ and $C$ (cf. e.{\,}g. \cite{Gc2.5}).
In particular, if $\omega$ is moderate on $\rr d$, then
$$
e^{-c|x|}\lesssim \omega (x)\lesssim e^{c|x|},
$$
for some constant $c>0$.

\par

The set of all moderate functions on $\rr d$
is denoted by $\mascP _E(\rr d)$. Furthermore, if $v$ in \eqref{moderate}
can be chosen as a polynomial, then $\omega$ is
called a weight of polynomial type, or polynomially moderated. We let
$\mascP (\rr d)$ be the set
of all polynomially moderated weights on $\rr d$. If $\omega (x,\xi
)\in \mascP _E (\rr {2d})$ is constant with respect to the
$x$-variable ($\xi$-variable), then we sometimes write $\omega (\xi )$
($\omega (x)$) instead of $\omega (x,\xi )$. In this case we consider
$\omega$ as an element in $\mascP _E(\rr {2d})$ or in $\mascP _E(\rr
d)$ depending on the situation.

\medspace

Let $\phi \in \maclS _s '(\rr d)$ be fixed. Then the \emph{short-time
Fourier transform} $V_\phi f$ of $f\in \maclS _s '
(\rr d)$ with respect to the \emph{window function} $\phi$ is
the Gelfand-Shilov distribution on $\rr {2d}$, defined by
$$
V_\phi f(x,\xi ) \equiv  (\mascF _2 (U(f\otimes \phi )))(x,\xi ) =
\mascF (f \, \overline {\phi (\cdo -x)})(\xi
),
$$
where $(UF)(x,y)=F(y,y-x)$. If $f ,\phi \in \maclS _s (\rr d)$, then it follows that
$$
V_\phi f(x,\xi ) = (2\pi )^{-d/2}\int f(y)\overline {\phi
(y-x)}e^{-i\scal y\xi}\, dy .
$$
We recall that the short-time Fourier transform is closely related to the
Wigner distribution, because
\begin{equation}\label{WigSTFTrelation}
W_{f,\phi} (x,\xi )=2^de^{2i\scal x\xi}V_{\check \phi}f(2x,2\xi ),
\end{equation}
which follows by elementary manipulations. Here $\check \phi (x)=\phi (-x)$.

\par

Let $\omega \in \mascP _E (\rr {2d})$, $p,q\in (0,\infty ]$
and $\phi \in \maclS _{1/2} (\rr d)\setminus 0$ be fixed. Then the mixed
Lebesgue space $L^{p,q}_{(\omega )}(\rr {2d})$ consists of
all measurable functions $F$ on $\rr {2d}$ such that
$\nm F{L^{p,q}_{(\omega )}}<\infty$. Here
\begin{equation}\label{Lpq1norm}
\nm F{L^{p,q}_{(\omega )}} \equiv \nm {F_{p,\omega}}{L^q},
\quad \text{where} \quad
F_{p,\omega}(\xi ) \equiv \nm {F(\cdo ,\xi )\omega (\cdo ,\xi )}{L^p}.
\end{equation}
We note that these quasi-norms might attain $+\infty$.

\par

The \emph{modulation space} $M^{p,q}_{(\omega )}(\rr d)$
is the quasi-Banach space which consist of
all $f\in \maclS _{1/2} '(\rr d)$ such that $\nm f{M^{p,q}_{(\omega
)}}<\infty$, where
\begin{equation}\label{modnorm}
\nm f{M^{p,q}_{(\omega )}}\equiv \nm {V_\phi f}{L^{p,q}_{(\omega
)}}.
\end{equation}
We remark that the definition of $M^{p,q}_{(\omega )}(\rr d)$
is independent of the choice of $\phi \in
\maclS _{1/2} (\rr d)\setminus 0$
and different $\phi$ gives rise to equivalent quasi-norms. (See Proposition
\ref{p1.4} below).

\par

For convenience we set $M^p _{(\omega )}= M^{p,p}_{(\omega
)}$. Furthermore we set $M^{p,q}=M^{p,q}_{(\omega )}$ when
$\omega \equiv 1$.

\par

The proof of the following proposition is omitted, since the results
can be found in \cite {Cor,Fe3,Fe4,FG1,FG2,FG3,GaSa,Gc2,To5,
To6, To7, To8, To11,To14}. Here we recall that $p,p'\in[1,\infty]$ satisfy
$\frac 1p+\frac 1{p'}=1$.

\par

\begin{prop}\label{p1.4}
Let $p,q,p_j,q_j\in (0,\infty ]$ for $j=1,2$, $r\le \min (p,q,1)$,
and $\omega ,\omega _1,\omega _2,v\in \mascP _E (\rr {2d})$
be such that $v$ is submultiplicative, $\omega$ is $v$-moderate and
$\omega _2\lesssim \omega _1$. Then the following is true:
\begin{enumerate}
\item[(1)] $f\in M^{p,q}_{(\omega )}(\rr d)$ if and only if
\eqref{modnorm} holds for any $\phi \in M^r_{(v)}(\rr d)\setminus
0$. Moreover, $M^{p,q}_{(\omega )}$ is a quasi-Banach space under
the quasi-norm
in \eqref{modnorm} and different choices of $\phi$ give rise to
equivalent quasi-norms. Furthermore, if $p,q\ge 1$, then $M^{p,q}_{(\omega )}$
is a Banach space;

\vrum

\item[(2)] if  $p_1\le p_2$ and $q_1\le q_2$  then
$$
\Sigma _1 (\rr d)\hookrightarrow M^{p_1,q_1}_{(\omega _1)}(\rr
d)\hookrightarrow M^{p_2,q_2}_{(\omega _2)}(\rr d)\hookrightarrow
\Sigma _1 '(\rr d)\text ;
$$

\vrum

\item[(3)] if in addition $p,q\ge 1$, then the $L^2$ product
$( \cdo ,\cdo )_{L^2}$ on $\maclS _{1/2}(\rr d)$
extends uniquely  to a continuous map from $M^{p,q}_{(\omega )}(\rr
n)\times M^{p'\! ,q'}_{(1/\omega )}(\rr d)$ to $\mathbf C$. On the
other hand, if $\nmm a = \sup \abp {(a,b)}$, where the supremum is
taken over all $b\in \maclS _{1/2} (\rr d)$ such that
$\nm b{M^{p',q'}_{(1/\omega )}}\le 1$, then $\nmm {\cdot}$ and $\nm
\cdot {M^{p,q}_{(\omega )}}$ are equivalent norms;

\vrum

\item[(4)] if $1\le p,q<\infty$, then $\maclS _{1/2} (\rr d)$ is dense in
$M^{p,q}_{(\omega )}(\rr d)$ and the dual space of $M^{p,q}_{(\omega
)}(\rr d)$ can be identified
with $M^{p'\! ,q'}_{(1/\omega )}(\rr d)$, through the $L^2$-form
$(\cdo  ,\cdo )_{L^2}$. Moreover, $\maclS _{1/2} (\rr d)$ is weakly dense
in $M^{p' ,q'}_{(\omega )}(\rr d)$ with respect to the $L^2$-form.
\end{enumerate}
\end{prop}

\par

Proposition \ref{p1.4}{\,}(1) allows us  be rather vague concerning
the choice of $\phi \in  M^r_{(v)}\setminus 0$ in
\eqref{modnorm}. For example, if $C>0$ is a
constant and $\mathscr A$ is a subset of $\maclS _{1/2} '$, then $\nm
a{M^{p,q}_{(\omega )}}\le C$ for
every $a\in \mathscr A$, means that the inequality holds for some choice
of $\phi \in  M^r_{(v)}\setminus 0$ and every $a\in
\mathscr A$. Evidently, a similar inequality is true for any other choice
of $\phi \in  M^r_{(v)}\setminus 0$, with  a suitable constant, larger
than $C$ if necessary.

\par

\begin{rem}\label{remGSmodident}
By Theorem 3.9 in \cite{To11} and Proposition \ref{p1.4} (2) it follows that
$$
\bigcap _{\omega \in \mascP _E}M^{p,q}_{(\omega )}(\rr d) = \Sigma _1(\rr d),
\quad
\bigcup _{\omega \in \mascP _E}M^{p,q}_{(\omega )}(\rr d) = \Sigma _1'(\rr d)
$$
More generally, let $s\ge 1$, $v_c(y,\eta ) = e^{c(|y|^{1/s}+|\eta |^{1/s})}$,
and let $\mathcal P$ respectively $\mathcal P _0$ be the set of all
$\omega \in \mascP _E(\rr {2d})$ such that
$$
\omega (x+y,\xi +\eta )\lesssim \omega (x,\xi )v_c(y,\eta ),
$$
for some $c>0$ respectively for every $c>0$. Then
\begin{alignat*}{2}
\bigcap _{\omega \in \mathcal P}M^{p,q}_{(\omega )}(\rr d) &= \Sigma _s(\rr d),
&\quad \phantom{\text{and}}\quad
\bigcup _{\omega \in \mathcal P}M^{p,q}_{(1/\omega )}(\rr d) &= \Sigma _s'(\rr d),
\\[1ex]
\bigcap _{\omega \in \mathcal P_0}M^{p,q}_{(\omega )}(\rr d) &= \maclS _s(\rr d),
&\quad  
\bigcup _{\omega \in \mathcal P_0}M^{p,q}_{(1/\omega )}(\rr d) &= \maclS _s'(\rr d),
\end{alignat*}
$$
\Sigma _s(\rr d)\hookrightarrow M^{p,q}_{(v_c )}(\rr d) \hookrightarrow  \maclS _s(\rr d)
\quad \text{and}\quad
\maclS _s'(\rr d) \hookrightarrow M^{p,q}_{(1/v_c )}(\rr d) \hookrightarrow \Sigma _s'(\rr d).
$$
(Cf. Proposition 4.5 in \cite{CPRT10}, Proposition 4. in \cite{GZ},
Corollary 5.2 in \cite{PT1} or Theorem 4.1 in \cite{Te2}. See also
\cite[Theorem 3.9]{To11} for an extension of these inclusions to broader classes of
Gelfand-Shilov and modulation spaces.)
\end{rem}

\par

\subsection{Schatten-von Neumann classes}

\par

Next we recall some properties on Schatten-von Neumann classes.
Let $\mascH _1$ and $\mascH _2$ be
Hilbert spaces, and let $T$ be a linear
map from $\mascH _1$ to $\mascH _2$. For
every integer $j\ge 1$, the \emph{singular number}  of $T$ of
order $j$ is given by
$$
\sigma _j(T) = \sigma _j(\mascH _1,\mascH _2,T)
\equiv \inf \nm {T-T_0}{\mascH _1\to \mascH _2},
$$
where the infimum is taken over all linear operators $T_0$ from $\mascH _1$
to $\mascH _2$ with rank at most $j-1$. Therefore, $\sigma _1(T)$
equals $\nm T{\mascH _1\to \mascH _2}$, and $\sigma _j(T)$ is
non-negative which decreases with $j$.

\par

For any $p\in (0,\infty ]$ we set
$$
\nm T{\mascI _p} = \nm T{\mascI _p(\mascH _1,\mascH _2)}
\equiv \nm { \{ \sigma _j(\mascH _1,\mascH _2,T) \} _{j=1}^\infty}{l^p}
$$
(which might attain $+\infty$). The operator $T$ is called a \emph{Schatten-von
Neumann operator} of order $p$ from $\mascH _1$ to $\mascH _2$, if
$\nm T{\mascI _p}$ is finite, i.{\,}e.
$\{ \sigma _j(\mascH _1,\mascH _2,T) \} _{j=1}^\infty$ should belong to $l^p$.
The set of all Schatten-von Neumann operators of order $p$ from
$\mascH _1$ to $\mascH _2$ is denoted by $\mascI _p =
\mascI _p(\mascH _1,\mascH _2)$. We note that
$\mascI _\infty(\mascH _1,\mascH _2)$ agrees with $\maclB (\mascH _1
,\mascH _2)$ (also in norms), the set of linear and bounded operators
from $\mascH _1$ to $\mascH _2$. If $p<\infty$, then
$\mascI _p(\mascH _1,\mascH _2)$ is contained in $\maclK(\mascH _1
,\mascH _2)$, the set of linear and compact operators from $\mascH _1$
to $\mascH _2$. The spaces $\mascI _p(\mascH _1,\mascH _2)$ for
$p\in (0,\infty ]$ and $\maclK(\mascH _1 ,\mascH _2)$ are quasi-Banach
spaces which are Banach spaces when $p\ge 1$. Furthermore,
$\mascI _2(\mascH _1,\mascH _2)$ is a Hilbert space and agrees with the
set of Hilbert-Schmidt operators from $\mascH _1$ to $\mascH _2$ (also in
norms). We set $\mascI  _p(\mascH )=\mascI  _p(\mascH ,\mascH )$.

\par

The set $\mascI _1(\mascH _1,\mascH _2)$ is the set of trace-class
operators from $\mascH _1$ to $\mascH _2$, and $\nm \cdo
{\mascI _1 (\mascH _1,\mascH _2)}$ coincide with the trace-norm. If in addition
$\mascH _1=\mascH _2=\mascH$, then the trace
$$
\operatorname{Tr}_\mascH (T) \equiv \sum _{\alpha} (Tf_\alpha ,f_\alpha)_{\mascH}
$$
is well-defined and independent of the orthonormal basis $\{ f_\alpha \}_{\alpha}$
in $\mascH$.

\par

Now let $\mascH _3$ be another Hilbert space
and let $T_k$ be a linear and continuous operator from
$\mascH _k$ to $\mascH _{k+1}$, $k=1,2$. Then we recall that the H{\"o}lder
relation 
\begin{equation}\label{SchattenComp}
\begin{gathered}
\nm {T_2\circ T_1}{\mascI _{r}(\mascH _1,\mascH _3)}\le 
\nm {T_1}{\mascI _{p_1}(\mascH _1,\mascH _2)}
\nm {T_2}{\mascI _{p_2}(\mascH _2,\mascH _3)}
\\[1ex]
\text{when}\quad
\frac 1{p_1}+\frac 1{p_2}=\frac 1r
\end{gathered}
\end{equation}
(cf. e.{\,}g. \cite{Si,To13}).

\par

In particular, the map $(T_1,T_2)\mapsto T_2^*\circ T_1$ is continuous from 
$\mascI _p(\mascH _1,\mascH _2)\times \mascI _{p'}(\mascH _1,\mascH _2)$
to $\mascI _1(\mascH _1)$, giving that
\begin{equation}\label{Eq:SchattenScalar}
(T_1,T_2)_{\mascI _2(\mascH _1,\mascH _2)}\equiv
\operatorname{Tr}_{\mascH _1}(T^*_2\circ T_1)
\end{equation}
is well-defined and continuous from $\mascI _p(\mascH _1,\mascH _2)\times
\mascI _{p'}(\mascH _1,\mascH _2)$ to $\mathbf C$. If $p=2$, then
the product, defined by \eqref{Eq:SchattenScalar} agrees with the scalar product in
$\mascI _2(\mascH _1,\mascH _2)$.

\par

The proof of the following result is omitted, since it can be found in e.{\,}g.
\cite{BS,Si}.

\par

\begin{prop}\label{Prop:SchattenDual}
Let $p\in [1,\infty]$, $\mascH _1$ and $\mascH _2$ be Hilbert spaces, and let $T$
be a linear and continuous map from $\mascH _1$ to $\mascH _2$. Then the
following is true:
\begin{enumerate}
\item if $q\in [1,p' ]$, then
$$
\nm T{\mascI _p(\mascH _1,\mascH _2)}
=\sup |(T,T_0)_{\mascI _2(\mascH _1,\mascH _2)}|,
$$
where the supremum is taken over all $T_0\in \mascI
_q(\mascH _1,\mascH _2)$ such that $\nm {T_0}{\mascI
_{p'}(\mascH _1,\mascH _2)}\le 1$;

\vrum

\item if in addition $p<\infty$, then the dual of $\mascI
_p(\mascH _1,\mascH _2)$ can be identified through the form
\eqref{Eq:SchattenScalar}.
\end{enumerate}
\end{prop}

\par

Later on we are especially interested of finding necessary and sufficient
conditions of symbols, in order for the corresponding pseudo-differential
operators to belong to $\mascI _p(\mascH _1,\mascH _2)$, where
$\mascH _1$ and $\mascH _2$ satisfy
$$
\Sigma _1(\rr d)\hookrightarrow \mascH _1,\mascH _2
\hookrightarrow \Sigma _1'(\rr d).
$$
Therefore, for such Hilbert spaces and $p\in (0,\infty ]$, let
\begin{align}
s_{A,p}(\mascH _1,\mascH _2) &\equiv
\sets {a\in \maclS _{1/2}'(\rr {2d})}{\op _A(a)\in \mascI
_p(\mascH _1,\mascH_2 )}\notag
\intertext{and}
\nm a{s_{A,p}(\mascH _1,\mascH _2)}
&\equiv \nm {\op _A(a)}{\mascI _p(\mascH _1,\mascH _2)}.
\label{SchattNormId}
\end{align}
Since the map $a\mapsto \op _A(a)$ is bijective from $\maclS _s'(\rr {2d})$
to the set of all linear and continuous operators from $\maclS _s(\rr d)$ to 
$\maclS _s'(\rr d)$, when $s\ge \frac 12$, it follows from the definitions
that the map $a\mapsto \op _A(a)$ restricts to a bijective and isometric map
from $s_{A,p}(\mascH _1,\mascH _2)$ to $\mascI _p(\mascH _1, \mascH _2)$.

\par

Usually it is assumed that $\mascH _1$ and $\mascH _2$ are tempered
in the sense of Definition 3.1 in \cite{To13}, or more restricted that
$\mascH _j=M^2_{(\omega _j)}(\rr d)$, for some $\omega _j\in \mascP _E(\rr {2d})$,
$j=1,2$. For conveniency we therefore set
$$
s_{A,p}(\omega _1,\omega _2) \equiv s_{A,p}(M^2_{(\omega _1)},M^2_{(\omega _2)}).
$$
We remark that the reader who is not interested in the most general setting may only consider
the case when $\mascH _j=M^2_{(\omega _j)}(\rr d)$, with $\omega _j\in \mascP _E(\rr {2d})$.
In this case, the $L^2$-dual of $\mascH _j$ is given by $M^2_{(1/\omega _j)}(\rr d)$.

%

\par

The latter bijectivity implies that Proposition \ref{Prop:SchattenDual} carries over
to analogous properties for $s_{A,p}(\mascH _1,\mascH _2)$ spaces. In the next section
we show that related results can be proved when the $s_{A,2}$ product and
$s_{p',A}(\mascH _1,\mascH _2)$ can be replaced by the $L^2$ product
and $s_{p',A}(\mascH _1',\mascH _2')$ when $\mascH _j$ are tempered with
$L^2$-duals $\mascH _j'$, $j=1,2$.

%
%
%
%

\par

%


\par

\section{Algebraic and continuity
properties}\label{sec2}

\par

In this section we deduce basic results for pseudo-differential operators
with symbols in modulation spaces, where the corresponding weights belong
to $\mascP _E$. The arguments are in general similar as corresponding
results in \cite{To6, To9}.

\par

The continuity results that we are focused on are especially Theorems
\ref{ThmA.2} and \ref{ThmA.3}. Here Theorem \ref{ThmA.2} deals with
pseudo-differential operators with symbols in modulation
spaces, which act on modulation spaces. Theorem \ref{ThmA.3} gives necessary
and sufficient conditions on symbols such that corresponding pseudo-differential
operators are Schatten-von Neumann operators of certain degrees.

\par

In Propositions \ref{PropA.4} and \ref{PropA.5} we deduce preparatory
results on Wigner distributions and pseudo-differential calculus in the context
of modulation space theory.

\par

In the last part we deduce composition properties for pseudo-differential
operators with symbols in modulation spaces. Especially we extend certain
results in \cite{CTW,HTW}.

\par

Let $s\ge \frac 12$ and let $K\in \maclS _s'(\rr {d_2+d_1})$.
Then $K$ gives rise to a linear and continuous operator $T=T_K$ from
$\maclS _s(\rr {d_1})$ to $\maclS _s'(\rr {d_2})$, defined by the formula
\begin{equation}\label{Eq:(A.1)}
Tf (x) = \scal {K(x,\cdo )}f ,
\end{equation}
which should be interpreted as \eqref{pre(A.1)} when $f\in
\maclS _s (\rr {d_1})$ and $g\in \maclS _s (\rr {d_2})$.

\par

%

\par

Before presenting the continuity properties of operators with kernels or symbols
in modulation spaces, we present the relations between the involved weight functions.
The weights $\omega ,\omega _0\in \mascP _E(\rr {4d})$ and
$\omega _1,\omega _2\in \mascP _E(\rr {2d})$ are in general related to each others
by the formulae
%
%
\begin{alignat}{2}
\frac {\omega _2(x,\xi )}{\omega _1(y,\eta )} &\lesssim \omega (x,y,\xi,-\eta ),
&\quad x,\xi &\in \rr {d_2},\ y,\eta \in \rr {d_1}
\label{Eq:(A.2)}
\intertext{or}
\frac {\omega _2(x,\xi )}{\omega _1(y,\eta )} &\asymp \omega (x,y,\xi,-\eta ),
&\quad x,\xi &\in \rr {d_2},\ y,\eta \in \rr {d_1},\tag*{(\ref{Eq:(A.2)})$'$}
\end{alignat}
and
\begin{align}
\omega (x,y,\xi ,\eta ) &\asymp  \omega _0(x-A(x-y),A^*\xi -(I-A^*)\eta ,
\xi +\eta ,y-x ),  
\notag
\\
&  \hspace{6cm} x,y,\xi ,\eta  \in \rr d,
\label{Eq:(A.3)}
\intertext{or equivalently,}
\omega _0(x,\xi ,\eta ,y) &\asymp  \omega (x-Ay,x+(I-A)y,\xi +(I-A^*) \eta ,-\xi +A^*\eta ),
\notag
\\
 &  \hspace{6cm} x,y,\xi ,\eta  \in \rr d.
\tag*{(\ref{Eq:(A.3)})$'$}
\end{align}
We note that \eqref{Eq:(A.2)} and \eqref{Eq:(A.3)} imply
\begin{equation}\label{Eq:(A.4)}
\frac {\omega _2(x,\xi  )}{\omega _1
(y,\eta )} \lesssim \omega _0( x-A(x-y),A^*\xi +(I-A^*)\eta ,\xi -\eta ,y-x ),
\end{equation}
and that  \eqref{Eq:(A.2)}$'$ and \eqref{Eq:(A.3)} imply
\begin{equation}\tag*{(\ref{Eq:(A.4)})$'$}
\frac {\omega _2(x,\xi  )}{\omega _1
(y,\eta )} \asymp \omega _0( x-A(x-y),A^*\xi +(I-A^*)\eta ,\xi -\eta ,y-x ),
\end{equation}

\par

The Lebesgue exponents of the modulation spaces should satisfy
conditions of the form
\begin{equation}\label{Eq:(A.5)}
\frac 1{p_1}-\frac 1{p_2}=\frac 1{q_1}-\frac 1{q_2} = 1-\frac 1p-\frac 1q,
\quad q\le p_2,q_2\le p ,
\end{equation}
or
\begin{equation}\label{Eq:(A.6)}
p_1\le p\le p_2,\quad  q_1\le \min (p,p') \quad \text{and}
\quad q_2\ge \max (p,p') .
\end{equation}

\par

The first result is essentially a fundamental kernel theorem for operators in the framework of
modulation space theory, and corresponds to Schwartz kernel theorem for. The
result goes back to \cite{Fe4} in the unweighted case (see also \cite{Gc2}).
The general case follows by combining Theorem A.1 in \cite{To13} with Proposition
\ref{PropA.5} below. The details are left for the reader.

\par

\begin{thm}\label{ThmA.1}
Let $\omega _j\in \mascP _E(\rr {2d_j})$ for
$j=1,2$ and $\omega \in \mathscr P_E(\rr {2d_2+2d_1})$ be
such that \eqref{Eq:(A.2)}$'$ holds. Also let $T$ be a linear and
continuous map from $\maclS _{1/2}(\rr {d_1})$ to
$\maclS _{1/2}' (\rr {d_2})$. Then the following conditions are
equivalent:
\begin{enumerate}
\item[(1)] $T$ extends to a continuous mapping from
$M^1_{(\omega _1)}(\rr {d_1})$ to $M^\infty _{(\omega _2)}
(\rr {d_2})$;

\vrum

\item[(2)] there is a unique $K\in M^\infty _{(\omega)}(\rr {d_2+d_1})$
such that \eqref{Eq:(A.1)} holds for every $f\in \maclS _{1/2}(\rr {d_1})$;

\vrum

\item[(3)] if in addition $d_1=d_2=d$, $A\in \GL (d,\mathbf R)$ and
\eqref{Eq:(A.3)} holds, then there is a
unique $a\in M^\infty _{(\omega _0)}(\rr {2d})$ such that $Tf=\op _A(a)f$
when $f\in \maclS _{1/2}(\rr d)$.
\end{enumerate}

\par

Furthermore, if {\rm{(1)--(2)}} are fulfilled, then $\nm T{M^1_{(\omega _1)}
\to M^\infty _{(\omega _2)}} \asymp \nm
K{M^\infty _{(\omega )}}$, and if in addition $d_1=d_2$, then $ \nm
K{M^\infty _{(\omega )}} \asymp \nm
a{M^\infty _{(\omega _0)}}$.
\end{thm}

\par

The next two results extend Theorems A.2 and A.3 in \cite{To13}.

\par

\begin{thm}\label{ThmA.2}
Let $A\in \GL (d,\mathbf R)$ and $p,q,p_j,q_j\in [1,\infty ]$ for
$j=1,2$, satisfy \eqref{Eq:(A.5)}. Also let $\omega _0\in \mascP _E(\rr {2d}\oplus \rr {2d})$
and $\omega _1,\omega _2\in \mascP _E(\rr {2d})$ satisfy \eqref{Eq:(A.4)}.
If $a\in M^{p,q}_{(\omega )}(\rr {2d})$,
then $\op _A(a)$ from $\maclS _{1/2}(\rr d)$ to $\maclS _{1/2}'(\rr d)$
extends uniquely to a continuous mapping from $M^{p_1,q_1}_{(\omega
_1)}(\rr d)$ to $M^{p_2,q_2}_{(\omega _2)}(\rr d)$, and
\begin{equation}\label{Eq:(A.7)}
\nm {\op _A(a)}{M^{p_1,q_1}_{(\omega _1)}\to
M^{p_2,q_2} _{(\omega _2)}} \lesssim \nm
a{M^{p,q} _{(\omega _0)}}.
\end{equation}

\par

Moreover, if in addition $a$ belongs to the closure of $\maclS _{1/2}$ under
the $M^{p,q}_{(\omega _0 )}$ norm, then $\op _A(a)\, :\, M^{p_1,q_1}_{(\omega _1)}\to
M^{p_2,q_2}_{(\omega _2)}$ is compact.
\end{thm}

\par

\begin{thm}\label{ThmA.3}
Let $A\in \GL (d,\mathbf R)$ and $p,q,p_j,q_j\in [1,\infty ]$ for
$j=1,2$, satisfy \eqref{Eq:(A.6)}. Also let $\omega _0\in \mascP _E(\rr {2d}\oplus \rr {2d})$
and $\omega _1,\omega _2\in \mascP _E(\rr {2d})$ satisfy \eqref{Eq:(A.4)}$'$.
Then
$$
M^{p_1,q_1}_{(\omega _0)}(\rr {2d}) \hookrightarrow
s_{A,p}(\omega _1,\omega _2)\hookrightarrow 
M^{p_2,q_2}_{(\omega _0)}(\rr {2d}).
$$
\end{thm}

\par

For the proofs we need the following extensions of Propositions 4.1
and 4.8 in \cite{To8}.

\par

\begin{prop}\label{PropA.4}
Let $A\in \GL (d,\mathbf R)$, and let $p_j,q_j,p,q\in (0,\infty ]$ be such that
$p\le p_j,q_j\le q$, for $j=1,2$, and 
\begin{equation}\label{Eq:(A.8)}
\frac 1{p_1} + \frac 1{p_2} = \frac 1{q_1}+\frac 1{q_2}=\frac 1p+\frac 1q.
\end{equation}
Also let $\omega _1,\omega _2\in \mascP _E(\rr {2d})$ and
$\omega \in \mascP _E(\rr {2d}\oplus \rr {2d})$ be such that
\begin{equation}\label{Eq:(A.9)}
\omega _0( x-A(x-y),A^*\xi +(I-A^*)\eta ,\xi -\eta ,y-x )
\lesssim
\omega _1(x,\xi )\omega _2(y,\eta ).
\end{equation}
Then the map $(f_1,f_2)\mapsto W_{f_1,f_2}^A$ from $\maclS _{1/2}'(\rr
d)\times \maclS _{1/2}'(\rr d)$ to $\maclS _{1/2}'(\rr {2d})$ restricts to a
continuous mapping from $M^{p_1,q_1}_{(\omega _1)}(\rr d)\times
M^{p_2,q_2}_{(\omega _2)}(\rr d)$ to $M^{p,q}_{(\omega _0
)}(\rr {2d})$, and
\begin{equation}\label{Eq:(A.10)}
\nm {W_{f_1,f_2}^A}{M^{p,q}_{(\omega _0)}}
\lesssim
\nm {f_1}{M^{p_1,q_1}_{(\omega _1)}} \nm
{f_2}{M^{p_2,q_2}_{(\omega _2)}}
\end{equation}
when $f_1,f_2\in \maclS _{1/2}'(\rr d)$.
\end{prop}

\par

\begin{prop}\label{PropA.5}
Let $p\in (0,\infty]$, $\omega _j\in \mascP _E(\rr {2d_j})$, $j=1,2$,
$\omega \in \mascP _E(\rr {2d_2+2d_1})$, and let $T$ be a
linear and continuous operator from $\maclS _{1/2}(\rr {d_1})$ to
$\maclS _{1/2}'(\rr {d_2})$ with distribution kernel
$K\in \maclS _{1/2}'(\rr {d_2+d_1})$. Then the following is true:
\begin{enumerate}
\item[(1)] if \eqref{Eq:(A.2)}$'$ holds, then $T\in \mathscr I_2(M^2_{(\omega
_1)}(\rr {d_1}),M^2_{(\omega _2)}(\rr {d_2}))$, if and only if $K\in M^2_{(\omega )}(\rr
{d_2+d_1})$, and then
\begin{equation}\label{Eq:(A.11)}
\nm T{\mathscr I_2} \asymp \nm K{M^2_{(\omega )}}\text ;
\end{equation}

\vrum

\item[(2)] if $d_1=d_2=d$, $A\in \GL (d,\mathbf R)$ and
$\omega _0\in \mascP _E(\rr {2d})$ satisfy
\eqref{Eq:(A.3)}$'$, $a\in \maclS _{1/2}'(\rr {2d})$ and $K=K_{a,A}$ is
given by \eqref{atkernel}, then $K\in M^p_{(\omega )}(\rr {2d})$
if and only if $a\in M^p_{(\omega _0)}(\rr {2d})$, and
$$
\nm K{M^p_{(\omega )}} \asymp \nm a{M^p_{(\omega _0)}}\text .
$$
\end{enumerate}
\end{prop}

\par

We need some preparations for the proofs.
First we note that \eqref{Eq:(A.9)} is the same as
\begin{equation}\tag*{(\ref{Eq:(A.9)})$'$}
\omega _0(x,\xi ,\eta ,y)
\lesssim
\omega _1(x-Ay,\xi +(I-A^*)\eta )\omega _2(x+(I-A)y,\xi -A^*\eta ).
\end{equation}

\par

\begin{lemma}\label{Lemma:STFTAWigner}
Let $A\in \GL (d,\mathbf R)$, $s\ge \frac 12$ $f,g\in \maclS _s'(\rr d)$,
$\phi ,\psi \in \maclS _s(\rr d)$, and let $\Phi = W_{\phi ,\psi}^A$. Then
\begin{multline*}
(V_\Phi W_{f,g}^A)(x,\xi ,\eta ,y)
\\[1ex]
=
e^{-i\scal y\xi}
(V_\phi f)(x-Ay,\xi -(A^*-I)\eta ) \overline{(V_\psi g)(x-(A-I)y,\xi -A^*\eta )}
\end{multline*}
\end{lemma}

\par

The proof of the preceding lemma follows by similar arguments as for
Lemma 14.5.1 in \cite{Gc2}. In order to be self-contained, we here present
the arguments.

\par

\begin{proof}
Again we consider the case when $f,g\in \maclS _s(\rr d)$, leaving
the modifications of the general case to the reader.

\par

Let
\begin{multline*}
H(x,y_1,y_2,y_3)
\\[1ex]
=
f(y_1+Ay_2)\overline{g(y_1+By_2)}
\overline {\phi (y_1-x+Ay_3)}\psi (y_1-x+By_3),
\end{multline*}
where $B=A-I$. Then Fourier's inversion formula gives
\begin{multline*}
(2\pi )^d(V_\Phi W_{f,g}^A)(x,\xi ,\eta ,y) 
\\[1ex]
=
(2\pi )^{-d}\iiiint H(x,y_1,y_2,y_3)
e^{-i(\scal {y_2+y}{\eta _1} +\scal {y_3}{\xi -\eta _1} +\scal {y_1}\eta )}
\, dy_1dy_2dy_3d\eta _1
\\[1ex]
=
e^{-i\scal y\xi} \iint H(x,y_1,y_2,y_2+y)
e^{-i(\scal{y_2}\xi +\scal {y_1}\eta )}\, dy_1dy_2 
\\[1ex]
= e^{-i\scal y\xi} \iint F(y_1+Ay_2,x-Ay)
\overline{G(y_1+By_2,x-By)}
e^{-i(\scal{y_2}\xi +\scal {y_1}\eta )}\, dy_1dy_2  ,
\end{multline*}
where
\begin{align*}
F(x,y) &= f(x)\overline{\phi (x-y)}
\intertext{and}
G(x,y) &= g(x)\overline{\psi (x-y)} .
\end{align*}

\par

By taking $(y_1+Ay_2,y_1+By_2)$ as new variables of integration, we
obtain
\begin{multline*}
(2\pi )^d(V_\Phi W_{f,g}^A)(x,\xi ,\eta ,y) 
\\[1ex]
=
e^{-i\scal y\xi} \iint F(z_1,x-Ay)
\overline{G(z_2,x-By)}
e^{-i(\scal {z_1-z_2}\xi +
\scal {z_1-A(z_1-z_2)}\eta)} \, dz_1dz_2
\\[1ex]
=
(2\pi )^d e^{-i\scal y\xi} I(x,\xi ,\eta ,y)J(x,\xi ,\eta ,y),
\end{multline*}
where
\begin{multline*}
I(x,\xi ,\eta ,y) = (2\pi )^{-\frac d2} \int 
f(z)\overline{\phi (z-(x-Ay))}e^{-i\scal z{\xi -B^*\eta}}\, dz 
\\[1ex]
= V_\phi f)(x-Ay,\xi -B^*\eta )
\end{multline*}
and
\begin{multline*}
J(x,\xi ,\eta ,y) = (2\pi )^{-\frac d2} \int 
\overline{g(z)}{\psi (z-(x-By))}e^{i\scal z{\xi -A^*\eta}}\, dz
\\[1ex]
= \overline{(V_\psi g)(x-By,\xi -A^*\eta )},
\end{multline*}
and the result follows by combining these equalities.
%
\end{proof}

\par

\begin{proof}[Proof of Proposition \ref{PropA.4}]
We only prove the result when $p,q<\infty$. The
straight-forward modifications to the cases $p=\infty$ or
$q=\infty$ are left for the reader. Let $\phi _1,\phi _2\in \Sigma
_1(\rr d)\setminus 0$, and let $\Phi =W_{\phi _1,\phi _2}^A$. Then
Fourier's inversion formula gives
\begin{multline*}
(V_\Phi (W_{f_1,f_2}^A))(x,\xi ,\eta ,y)
\\[1ex]
= e^{-i\scal y\xi }F_1(x-Ay,\xi
+(I-A^*)\eta )\overline{F_2(x+(I-A)y,\xi -A^*\eta )},
\end{multline*}
where $F_j=V_{\phi _j}f_j$, $j=1,2$. By applying the $L^{p,q}_{(\omega )}$-norm
on the latter equality, and using \eqref{Eq:(A.9)}$'$, it follows from
Minkowski's inequality that
$$
\nm {W_{f_1,f_2}^A}{M^{p,q}_{(\omega _0)}} \lesssim \big (\nm
{G_1*G_2}{L^{r}}\big )^{1/p} \le \Big (\int H(\eta )\, d\eta \Big
)^{1/q},
$$
where $G_j=(F_j\omega _j)^p$, $r=q/p\ge 1$ and
$$
H(\eta ) = \Big (\int \Big (\int \Big (\int G_1(y-x,\eta -\xi
)G_2(x,\xi )\, dx\Big )^{r}\, dy\Big )^{1/r}\, d\xi \Big )^{r}.
$$
Now let $r_j,s_j\in [1,\infty ]$ for $j=1,2$ be chosen such that
$$
\frac 1{r_1}+\frac 1{r_2} = \frac 1{s_1}+\frac 1{s_2}=1+\frac 1r .
$$
Then Young's inequality gives
\begin{equation*}
H(\eta )\le \Big ( \int \nm {G_1(\cdo ,\eta -\xi )}{L^{r_1}}\nm
{G_2(\cdo ,\xi )}{L^{r_2}}\, d\xi \Big )^r
\end{equation*}
Hence another application of Young's inequality gives
$$
\nm {W_{f_1,f_2}^A}{M^{p,q}_{(\omega _0)}} \lesssim \Big (\int
H(\eta )\, d\eta \Big )^{1/q}\lesssim \big (\nm {G_1}{L^{r_1,s_1}}\nm
{G_2}{L^{r_2,s_2}}\big )^{1/p}
$$
By letting $p_j=pr_j$ and $q_j=qs_j$, the last inequality gives \eqref{Eq:(A.10)}.
The proof is complete.
\end{proof}

\par

\begin{proof}[Proof of Proposition \ref{PropA.5}]
We only prove (2), since (1) is a restatement of Proposition A.5 (2) in
\cite{To13}.

\par

Let $\Phi ,\Psi \in \maclS _{1/2}(\rr {2d})\setminus 0$ be such that
$$
\Phi (x,y) =(\mathscr F_2\Psi )(x-A(x-y),x-y).
$$
Then it follows by straight-forward applications of Fourier's inversion formula
that
$$
|(V_\Phi K_{a,A})(x,y,\xi ,\eta )|
\asymp
|(V_\Psi a)(x-A(x-y),A^*\xi +(A^*-I)\eta ,\xi +\eta ,y-x)|.
$$
The assertion now follows by applying the
$L^p _{(\omega )}$ quasi-norm on the last equality, and using the fact that 
modulation spaces are independent of the choice of window functions 
in the definition of the modulation space quasi-norm (cf. Propositions 3.1 and
3.4 in \cite{To14}).
%
%
%
\end{proof}

\par

\begin{proof}[Proof of Theorem \ref{ThmA.1}]
The equivalence between (1) and (2) follows from Theorem A.1 in \cite{To13},
and the equivalence between (2) and (3) follows immediately from Proposition
\ref{PropA.5}.
\end{proof}

\par

\begin{proof}[Proof of Theorem \ref{ThmA.2}]
The conditions on $p_j$ and $q_j$ implies that
$$
p'\le p_1,q_1,p'_2,q'_2\le q',\quad
\frac 1{p_1}+\frac 1{p'_2}=\frac 1{q_1}+\frac 1{q'_2}=\frac 1{p'}+\frac 1{q'}.
$$
Hence Proposition \ref{PropA.4}, and \eqref{Eq:(A.4)} show that
$$
\nm {W^A_{g, f}}{M^{p',q'}_{(1/\omega )}}\lesssim
\nm f{M^{p_1,q_1}_{(\omega _1)}}\nm g{M^{p'_2,q'_2}_{(1/\omega _2)}}
$$
when $f\in M^{p_1,q_1}_{(\omega _1)}(\rr d)$ and $g\in
M^{p'_2,q'_2}_{(1/\omega _2)}(\rr d)$.

\par

The continuity is now an immediate consequence of \eqref{tWigpseudolink}
and Proposition \ref{p1.4} (4), except for the case $p=q'=\infty$,
which we need to consider separately.

\par

Therefore assume that $p=\infty$, and $q=1$, and let $a\in
\maclS _{1/2}(\rr {2d})$. Then $p_1=p_2$ and $q_1=q_2$, and it follows
from Proposition \ref{PropA.4} and the first part of the proof that
$W_{g,f}\in M^{1,\infty}_{(1/\omega _0)}$, and that \eqref{Eq:(A.7)} holds.
In particular,
$$
|(\op _A(a)f,g))|\lesssim \nm f{M^{p_1,q_1}_{(\omega _1)}}
\nm g{M^{p_1',q_1'}_{(1/\omega _2)}},
$$
and the result follows when $a\in \maclS _{1/2}$. The result now
follows for general $a\in M^{\infty ,1}_{(\omega _0)}$, by taking a
sequence $\{ a_j \} _{j\ge 1}$ in $\maclS _{1/2}$, which converges
narrowly to  $a$. (For narrow convergence see Theorems 4.15 and
4.19, and Proposition 4.16 in \cite{To11}).

\par

It remains to prove that if $a$ belongs to the closure of $\maclS _{1/2}$ under
$M^{p,q}_{(\omega )}$ norm, then $\op _A(a)\, :\, M^{p_1,q_1}_{(\omega
_1)}\to M^{p_2,q_2}_{(\omega _2)}$ is compact. As a consequence of
Theorem \ref{ThmA.3}, it follows that $\op _A(a_0)$ is compact when $a_0\in
\maclS _{1/2}$, since $\maclS _{1/2}\hookrightarrow M^1_{(\omega _0)}$
when $\omega _0\in \mascP _E$, and that every trace-class operator
is compact. The compactness of $\op _A(a)$ now follows by approximating
$a$ with elements in $\maclS _{1/2}$. The proof is complete.
\end{proof}

\par

For the proof of Theorem \ref{ThmA.3} we need the following extension of
Theorem 4.12 in \cite{To13}.

\par

\begin{prop}\label{Prop:SchattenIdenti}
Let $A\in \GL (d,\mathbf R)$, $p\in [1,\infty )$ and that $\mathscr H
_1,\mathscr H_2$ are tempered Hilbert spaces on $\rr d$. Then the
$L^2$ form on $\mathscr S(\rr {2d})$ extends
uniquely to a duality between $s_{A,p}(\mathscr H_1,\mathscr H_2)$ and
$s_{A,p'}(\mathscr H_1',\mathscr H_2')$, and the dual space for
$s_{A,p}(\mathscr H_1,\mathscr H_2)$ can be identified with
$s_{A,p'}(\mathscr H_1',\mathscr H_2')$ through this form. Moreover, if
$\ell \in s_{A,p}(\mathscr H_1,\mathscr H_2)^*$ and $a\in
s_{A,p'}(\mathscr H_1',\mathscr H_2')$ are such that $\overline{\ell
(b)}=(a,b)_{L^2}$ when $b\in s_{A,p}(\mathscr H_1,\mathscr H_2)$, then
$$
\nmm \ell = \nm a{s_{A,p'}(\mathscr H_1',\mathscr H_2')}.
$$
\end{prop}

\par

\begin{proof}
The result follows from Theorem 4.12 in the case $A=0$. For general
$A$, the result now follows from Proposition \ref{Prop:CalculiTransfer}
and the fact that $e^{i\scal {AD_\xi}{D_x}}$ is unitary on $L^2(\rr {2d})$.
\end{proof}

\par

\begin{proof}[Proof of Theorem \ref{ThmA.3}]
The first inclusion in
$$
M^{\infty ,1}_{(\omega _0)} \hookrightarrow s_{A,\infty}(\omega _1,\omega _2)
\hookrightarrow M^\infty _{(\omega _0)}
$$
follows from Theorem \ref{ThmA.2}, and the second one
from Proposition \ref{p1.4} (2) and Theorem \ref{ThmA.1}.

\par

By Propositions \ref{p1.4} (3), \ref{Prop:SchattenDual} and
\ref{Prop:SchattenIdenti}, \eqref{SchattNormId}, and duality,
the latter inclusions give
$$
M^{1}_{(\omega _0)} \hookrightarrow s_{A,1}(\omega _1,\omega _2)
\hookrightarrow M^{1,\infty} _{(\omega _0)},
$$
and we have proved the result when $p=1$ and when $p=\infty$.
Furthermore, by Proposition \ref{PropA.5} we have
$M^{2}_{(\omega _0)} = s_{A,2}(\omega _1,\omega _2)$, and
the result also holds in the case $p=2$. The result now follows
for general $p$ from these cases and interpolation. (See e.{\,}g.
Proposition 5.8 in \cite{To11}.) The proof
is complete.
\end{proof}

\par

The next result shows that the operator $e^{i\scal{AD_\xi}{D_x}}$ is bijective
between suitable modulation spaces. (See also \cite{To5,To9,To13} for similar results in
restricted cases.)

\par

\begin{prop}\label{Prop:ExpOpSTFT}
Let $s\ge \frac 12$, $A\in \GL (d,\mathbf R)$, $p,q\in (0,\infty ]$,
$\phi ,a\in \maclS _s(\rr {2d})$ and let $T_A = e^{i\scal{AD_\xi}{D_x}}$.
Then
\begin{equation}\label{Eq:ExpOpSTFT}
(V_{T_A\phi}(T_Aa)) (x,\xi ,\eta ,y) = e^{i\scal {Ay}\eta}
(V_\phi a)(x+Ay,\xi +A^*\eta ,\eta ,y).
\end{equation}
Furthermore, if $\omega \in \mascP _E(\rr {4d})$ and
$$
\omega _A(x,\xi ,\eta ,y) = \omega (x+Ay,\xi +A^*\eta ,\eta ,y),
$$
then $T_A$ from $\maclS _s(\rr {2d})$ to $\maclS _s(\rr {2d})$
extends uniquely to a homeomorphism from $M^{p,q}_{(\omega )}(\rr {2d})$
to $M^{p,q}_{(\omega _A)}(\rr {2d})$, and
\begin{equation}\label{Eq:ExpOpModSp}
\nm {T_Aa}{M^{p,q}_{(\omega _A)}} \asymp
\nm a{M^{p,q}_{(\omega )}}.
\end{equation}
\end{prop}

\par

\begin{proof}
The formula \eqref{Eq:ExpOpSTFT} follows by a straight-forward
application of Fourier's inversion formula (cf. \cite[Proposition 1.5]{To6}
and its proof). The estimate \eqref{Eq:ExpOpModSp} then follows
by first choosing $\phi$ in $\Sigma _1(\rr {2d})\setminus 0$ in
\eqref{Eq:ExpOpSTFT}, then multiplying this equation by
$\omega _A$ and thereafter applying the mixed $L^{p,q}$ quasi-norm.
This gives the result.
\end{proof}

\par

\subsection{Composition properties}

\par

Let $s\ge \frac 12$, $A\in \GL (d,\mathbf R)$ and
$a_1,\dots ,a_N\in \maclS _s(\rr {2d})$. Then the $N$-linear
product
\begin{equation}\label{Eq:N-LinPsProd}
(a_1,\dots ,a_N)\mapsto a_1\wpr _A \cdots \wpr _A a_N
\end{equation}
is defined by the formula
$$
\op _A(a_1)\circ \cdots \circ \op _A(a_N)
=
\op _A(a_1\wpr _A \cdots \wpr _A a_N).
$$
The $N$-linear product \eqref{Eq:N-LinPsProd} extends in different ways.
Here below we give examples on extensions in the framework of
modulation space theory.

\par

By a straight-forward application of Proposition
\ref{Prop:CalculiTransfer} we have
%
\begin{multline}\label{Eq:CalculiTransfComp}
a_1\wpr _A \cdots \wpr _A a_N
\\[1ex]
=
T_{A-B}^{\ -1}\big (
(T_{A-B}a_1)\wpr _B \cdots \wpr _B
(T_{A-B}a_N) \big ),\quad
T_{A}\equiv e^{i\scal {AD_\xi}{D_x}},
\end{multline}
for $A,B\in \GL (d,\mathbf R)$ and suitable $a_1,\dots ,a_N$.

\par

We have now the following results on compositions. Here
it is assumed that the weight functions should obey
\begin{multline}\label{weightcondAcalc}
1 \lesssim \omega _0(T_A(X_N,X_0))\prod _{j=1}^N
\omega _j(T_A(X_{j},X_{j-1})),
\quad X_0,\dots ,X_N \in \rr {2d},
\end{multline}
where
\begin{multline}\label{Ttdef}
T_A(X,Y) =(y+A(x-y),\xi+A^*(\eta -\xi ) ,\eta -\xi , x-y),
\\[1ex]
X=(x,\xi )\in \rr {2d},\ Y=(y,\eta )\in \rr {2d}.
\end{multline}
As in \cite{CoToWa} we also let $\masfR _N(\mabfp )$ with
$\mabfp =(p_0,\dots ,p_N)\in [1,\infty ]^{N+1}$
be the H{\"o}lder-Young functional
\begin{equation}\label{HYfunctional}
\begin{aligned}
\masfR _N(\mabfp ) &= ({N-1})^{-1}\left ({\sum _{j=0}^N\frac
1{p_j}-1}\right ),
\\[1ex]
\mabfp  &= (p_0,p_1,\dots ,p_N)\in [1,\infty ]^{N+1}.
\end{aligned}
\end{equation}

\par

\begin{thm}\label{Thm:Alg1}
Let $s\ge \frac 12$, $A\in \GL (d,\mathbf R)$, $p_j,q_j\in [1,\infty ]$,
$j=0,1,\dots , N$, and suppose
\begin{equation}\label{pqconditions}
\max \left ( \masfR _N(\mabfq ') ,0 \right )
\le  \min _{j=0,1,\dots ,N} \left ( \frac 1{p_j},\frac 1{q_j'},
\masfR _N(\mabfp )\right ).
\end{equation}
Let $\omega _j \in \mascP _E(\rr {4d})$, $j=0,1,\dots ,N$, and suppose
\eqref{weightcondAcalc} holds. Then the map
\eqref{Eq:N-LinPsProd} from $\maclS _s(\rr {2d}) \times \cdots \times
\maclS _s(\rr {2d})$ to $\maclS _s(\rr {2d})$
extends uniquely to a continuous and associative map from $M ^{p_1,q_1}
_{(\omega _1)}(\rr {2d}) \times \cdots \times M ^{p_N,q_N}
_{(\omega _N)}(\rr {2d})$ to $M ^{p_0',q_0'} _{(1/\omega _0)}(\rr {2d})$.
\end{thm}

\par

\begin{thm}\label{Thm:Alg2}
Let $s\ge \frac 12$, $A\in \GL (d,\mathbf R)$, $p_j,q_j\in [1,\infty ]$,
$j=0,1,\dots , N$, and suppose
\begin{equation}\label{pqconditions3}
\masfR _N(\mabfp ) \ge 0
\quad \text{and}\quad
\frac 1{q_j'}\le  \frac 1{p_j} \leq \frac12.
\end{equation}
Let $\omega _j \in \mascP _E(\rr {4d})$, $j=0,1,\dots ,N$, and suppose
\eqref{weightcondAcalc} holds. Then the map
\eqref{Eq:N-LinPsProd} from $\maclS _s(\rr {2d}) \times \cdots \times
\maclS _s(\rr {2d})$ to $\maclS _s(\rr {2d})$
extends uniquely to a continuous and associative map
from $M ^{p_1,q_1} _{(\omega _1)}(\rr {2d}) \times \cdots
\times M ^{p_N,q_N} _{(\omega _N)}(\rr {2d})$ to
$M^{p_0',q_0'} _{(1/\omega _0)}(\rr {2d})$.
\end{thm}

\par

\begin{proof}[Proof of Theorems \ref{Thm:Alg1} and \ref{Thm:Alg2}]
The result follows immediately from Theorems 0.1$'$ and 2.9 in
\cite{CoToWa} in the Weyl case, $A=\frac 12I$. For general $A$
the result now follows by from the Weyl case and straightforward
applications of \eqref{Eq:CalculiTransfComp} and
Proposition \ref{Prop:ExpOpSTFT}. 
\end{proof}

\par

\section{An idea of quantization}\label{sec3}

\par

In this section we make a suitable average of $\op_A(a)$ with respect to the matrix
$A$ to deduce certain types of operators, related to the symbols $a$, and which might
be of interests in quantizations.

\par

We recall that a quantization is a rule which takes an observable
$a(x,\xi )$ in classical mechanics to the corresponding observable
$\op (a)$ in quantum mechanics. Usually, $a$ is a suitable
function or distribution defined on the phase space $\rr {2d}$, and
$\op (a)$ is an operator which acts on suitable dense subspaces
of $L^2(\rr d)$.

\par

A common quantization is the Weyl quantization, $a\mapsto \op ^w(a)$,
explained earlier. Another quantization rule is the Born-Jordan quantization,
$a\mapsto \op _{\BJ} (a)$, where
\begin{equation}\label{Eq:BJDef}
\op _{\BJ} (a) = \int _0^1 \op _t(a)\, dt = \int _{-1/2}^{1/2} \op _{(t+\frac 12)\cdot I}(a)\, dt,
\end{equation}
provided the right-hand side makes sense. By straight-forward computations
it follows that
$$
\op _{BJ} (a) = \op ^w(\Phi *a),\qquad \Phi (x,\xi ) = \sinc (\scal x\xi /2).
$$

\par

We shall now consider other candidates of quantization, where the average
on the right-hand side in \eqref{Eq:BJDef} over all matrices $t\cdot I$,
$0\le t\le 1$, is replaced by averages over all $r\cdot U$, $U\in \UN 
(d,\mathbf R)$, and $r$ is fixed or taken over certain interval
$I\subseteq \mathbf R_+$. Here $\UN (d,\mathbf R)=\UN _d$ is the set of all
$d\times d$ orthonormal matrices with entries in $\mathbf R$.

\par

More precisely, for fixed $r\ge 0$, we let
\begin{align*}
\op _{r,\UN}(a) &\equiv \left ( \int _{U\in \UN _d} \, dU \right )^{-1}
\int _{U\in \UN _d} \op _{rU+\frac 12I}(a)\, dU,
\\[1ex]
\op _{r,\UN}^0(a) &\equiv r^{-1}\int _0^r \op _{t,\UN}(a)\, dt,
\\[1ex]
\psi _d(\rho )
&\equiv
\begin{cases}
(-i)^{\frac {d-2}2}\Gamma (\frac d2)\cdot
\displaystyle{\frac {J_{(d-2)/2}(i\rho )}{\rho ^{\frac {d-2}2}}}, & d>1
\\[1ex]
\cosh (\rho ), & d=1,
\end{cases}
\intertext{and}
\psi _{0,d}(r) &\equiv \int _0^r \psi _d(t)\, dt
\end{align*}
where $J_\nu$ is the Bessel function of order $\nu \in \mathbf R$.

\par

By straight-forward computations it follows that
\begin{alignat*}{3}
\op _{r,\UN}(a) &= \op ^w(a*\Psi _r),&
\quad &\text{where} &\quad \Psi _r(x,\xi ) &= \psi _d(r|x|\, |\xi |),
\intertext{and}
\op _{r,\UN}^0(a) &= \op ^w(a*\Psi _r^0),&
\quad &\text{where} & \quad \Psi _r^0(x,\xi ) &= \frac {\psi _{0,d} (r|x|\, |\xi |)}{r |x|\, |\xi |}.
\end{alignat*}

\par

\end{document}